\documentclass{article}

\usepackage{graphicx} 
\usepackage{cite}
\usepackage{authblk}
\usepackage{amsthm}
\usepackage{amsmath}
\usepackage{amssymb}
\usepackage[left=3cm, right=3cm, lines=45, top=1.0in, bottom=1.0in]{geometry}
\usepackage[colorlinks=true]{hyperref}
\date{}
\title{Nadirashvili's conjecture for elliptic PDEs and its applications}
\author{Jiahuan Li, Junyuan Wang, Zhichen Ying}

\newcommand{\keywords}[1]{\par\quad\textbf{Keywords:} #1}

\begin{document}

\maketitle
\newtheorem{theorem}{Theorem}[section] 
\newtheorem{definition}[theorem]{Definition} 
\newtheorem{lemma}[theorem]{Lemma}
\newtheorem{corollary}[theorem]{Corollary}
\newtheorem{example}[theorem]{Example}
\newtheorem{proposition}[theorem]{Proposition}
\newtheorem{conjecture}[theorem]{Conjecture}
\newtheorem{remark}[theorem]{Remark}

\begin{abstract}
    In this article, we investigate the conjecture posed by Nadirashvili in \cite{nadirashvilli1997geometry}. It states that if a harmonic function has bounded nodal volume in the unit ball, then the supermum over the half-ball can be bounded by a finite sum of derivatives at the center. The main tool in this paper is the lower bound of nodal sets, which is first proved in \cite{MR3739232}. Also we combine the propagation of smallness property and elliptic estimates to give a positive answer to this conjecture. In fact, we can extend this conjecture to general elliptic PDEs with smooth coefficients and also obtain a weak verison for less regular coefficients. Finally, we give several applications of this conjecture.
\end{abstract}

\keywords{Elliptic PDEs, nodal sets, Nadirashvili's conjecture}

\numberwithin{equation}{section}

\section{Introduction}

The nodal sets of PDEs have been extensively investigated in the past decades, especially the nodal sets of Laplace eigenfunctions, the solutions to linear elliptic and parabolic PDEs. There are two main conjectures in this field, one of them is the following Yau's conjecture.
\begin{conjecture}[\cite{MR645728}]\label{Yau}
    Let $(M^{n},g)$ be a closed manifold and $u$ be a nonconstant eigenfunction that satifies $\Delta_{g} u+\lambda u=0$ with $\lambda>0$. Then
    \begin{equation}
        C^{-1}(M,g)\sqrt{\lambda}\leq \mathcal{H}^{n-1}(Z(u))\leq  C(M,g)\sqrt{\lambda}.
    \end{equation}
\end{conjecture}
And another important conjecture comes from F. H.  Lin. He considered the solutions to elliptic equations of divergence form. Now considering the following equation 
\begin{equation}\label{divergence form of elliptic equation}
   \partial_{i}(a^{ij}(x)\partial_{j}u)+b^{i}(x)\partial_{i}u+c(x)u=0
\end{equation}
in $B(0,1)$ with 
\begin{equation}\label{elliptic condition}
    \Lambda^{-1}|\xi|^{2}\leq a^{ij}\xi_{i}\xi_{j}\leq \Lambda|\xi|^{2},\quad |b^{i}|,|c|\leq \Lambda
\end{equation}
and
\begin{equation}\label{lip}
    |a^{ij}(x)-a^{ij}(y)|\leq \Lambda|x-y|.
\end{equation}
\begin{conjecture}[\cite{MR1090434}]\label{Lin's conj}
    Let $u$ be a solution to \eqref{divergence form of elliptic equation} with coefficients condition \eqref{elliptic condition} and \eqref{lip}. Then the Hausdorff measure of nodal set satisfies
    \begin{equation}
        \mathcal{H}^{n-1}(Z(u)\cap B(0,1/2))\leq C(n,\Lambda)\beta_{u}(0,1)
    \end{equation}
    and  the Hausdorff measure of singular set satisfies
    \begin{equation}
        \mathcal{H}^{n-2}(S(u)\cap B(0,1/2))\leq C(n,\Lambda)\beta_{u}^{2}(0,1).
    \end{equation}
\end{conjecture}
\indent Here $\beta_{u}(0,1)$ is the frequency function defined in Definition \ref{frequency}.\\
\indent In the following, we review some important developments about these two conjectures. For Conjecture \ref{Yau}, an important breakthrough comes from Donnelly and Fefferman \cite{MR943927}. They have showed that the frequency of $u$ is bounded from above by $C(M,g)\sqrt{\lambda}$ and gave a positive answer to Yau's conjecture for the case with analytic metric.\\
\indent Without the analytic condition, there are also many results. For $n=2$, Bruning \cite{Bruning} derived the optimal lower bound $\mathcal{H}^{1}(Z(u))\geq C(M,g)\sqrt{\lambda}$. For $n=3$, Colding-Minicozz\cite{Colding-Minicozzi} got a uniform lower bound $\mathcal{H}^{2}(Z(u))\geq C(M,g)>0$. For general dimension, Sogge and Zelditch proved the lower bound $\mathcal{H}^{n-1}({Z(u)})\geq C\lambda^{\frac{3-n}{4}}$ in \cite{MR3091613}. For the upper bound, Donnely and Fefferman\cite{Donnely2} also obtained $\mathcal{H}^{1}(Z(u))\leq C\lambda^{3/4}$ in case $n=2$. In addition, Nadirashvili also derived an upper bound $\lambda log\lambda$ when $n=2$. Recently, Logunov proved the sharp lower bound of Yau's conjecture and gave a polynomial upper bound in \cite{MR3739231} and \cite{MR3739232}. One can also see \cite{logunov2019review} for more details about Yau's conjecture.  When considering the boundary, Logunov, Malinnikova, Nadirashvili and Nazarov \cite{MR4356702} obtained the sharp upper bound of the nodal sets of Dirichlet Laplace eigenfunction, and later Chen and Yang \cite{chen2024sharp} proved the sharp upper bound of nodal sets of Neumann Laplace eigenfunctions. Besides, in a more recent paper \cite{MR4814921}, Logunov, Priya, M. E., and Sartori showed an almost sharp lower bound of nodal volume with respect to doubling index. We also refer the reader to other types of eigenfunctions and estimates of nodal sets such as \cite{MR3928756}, \cite{MR4403241}, \cite{liu2023measure}. \\
 
Consider the general elliptic equation \eqref{divergence form of elliptic equation} with Lipschitz coefficients, Hardt and Simon obtained an explicit exponential bound for the nodal set in \cite{MR1010169}.  In 1994, Han and Lin \cite{Hanlin94} also proved this result in a different method. For singular sets and critical sets, in 1994, Han \cite{MR1305956} studied the structure of singular sets to elliptic PDEs. In 1998, Han, Hardt, and Lin \cite{HanHardtLin} have shown the finiteness of $(n-2)$ Hausdorff measure of singular sets under smooth coefficients. Furthermore, they also studied the singular sets of higher elliptic PDEs in \cite{MR2015412}.\\
 \indent In 2015, Cheeger, Naber, and Valtorta used quantitative stratification to give a weak estimate of the upper bound of the Hausdorff measure and the Minkowski content of critical sets in \cite{MR3298662} with only Lipschitz coefficients. Later, Naber and Valtorta based on the quantitative uniqueness of the tangent map and the quantitative cone-splitting principle to give a strong estimate for singular sets in \cite{MR3688031}. For some new developments, one can see \cite{huang2023volume} for elliptic equations with only H\"{o}lder coeffients and see \cite{huang2024nodal} for the nodal sets at the time slice of the solutions to parabolic equations with Lipschitz coefficients. \\

The above is about the work on the estimates of nodal sets. However, a \textbf{natural question} is how to use the nodal sets to obtain the information of the solutions. In fact, Nadirashvili has the following conjecture.
\begin{conjecture}
    There are constants $C$ and $N$ depending only on $n$ and $K$ such that if $u$ is a harmonic function in the ball $B(0,1)\subset\mathbb{R}^{n}$ and $\mathcal{H}^{n-1}(Z(u)\cap B(0,1))\leq K$ then
    \begin{equation} \label{1.1}
        \max_{B(0,1/2)}|u|\leq C\sum_{|\alpha|=0}^{N}|D^{\alpha}u(0)|.
    \end{equation}
\end{conjecture}
This conjecture indicates the interaction between the nodal sets and the solution itself. In addition, it is interesting to mention that Conjecture 1.3 can be viewed as a bridge between the strong unique continuation and a generalized Harnack inequality, which are two critical cases of this conjecture. In detail, when $K=\infty$, it is a version of the strong unique continuation. Also when $K=0$, it can be regarded as a generalized version of Harnack inequality, and this is actually a direct corollary of the following weaker-form conjecture.
\begin{conjecture}
    There exists a universal constant $\delta>0$ such that if $u$ is a harmonic function in $B(0,1)\subset \mathbb{R}^{n}$ and $u(0)=0$ then\begin{equation} \label{1.2}
        \mathcal{H}^{n-1}(Z(u)\cap B(0,1))\geq \delta.
    \end{equation}
\end{conjecture}
\begin{remark}
    In Section 6, we explain that Conjecture 1.3 can imply Conjecture 1.4.
\end{remark}
 The second conjecture gives a lower bound estimate of nodal sets of harmonic functions, and is proved by Alexander Logunov in \cite{MR3739232}. In this paper, \textbf{our main goal} is to investigate what kind of elliptic PDE, Conjecture 1.3 holds. And we also extend this conjecture to general elliptic PDEs with smooth coefficients and also obtain a weak verison for less regular coefficients.  Besides, there are also some articles regarding the interaction between nodal sets and the function itself. For example, Logunov and Malinnikova considered harmonic functions that sharing the same nodal sets in \cite{MR3318150} and \cite{MR3540456}. In addition, F. Lin and Z. Lin studied more general cases in \cite{MR4514975}. These articles inspired us to give some applications.\\

 In the following theorems, we will consider
\begin{equation}
    \mathrm{div}(A\nabla u)=0.
\end{equation}
for two different types of $A$ with different regularity. We denote by $\mathcal{A}_{1}$ the set of matrix $A$ such that it is a matrix with $C^{\infty}$ smooth coefficients and uniform ellipticity, i.e., there exists a constant $\lambda>0$ such that for all $\xi\in \mathbb{R}^{n}$, $A$ satisfies
\begin{equation}\label{uniform elliptic}
    \lambda^{-1}|\xi|^{2}\leq \xi^{T}A\xi\leq \lambda|\xi|^{2}.
\end{equation}
Next, we denote by $\mathcal{A}_{2}$ the set of matrix $A$ such that it is a matrix with $C^{1}$ smooth coefficients and with the same uniform elliptic condition \eqref{uniform elliptic}. Also, all components of $A$ have a Lipschitz constant smaller than $\lambda$.
One of our main results is the following.
\begin{theorem} \label{thm1.3}
    Let $u$ be a solution to $\mathrm{div}(A\nabla u)=0$ in $B(0,1)\subset \mathbb{R}^{n}$, where $A\in \mathcal{A}_{1}$ and $\mathcal{H}^{n-1}(Z(u)\cap B(0,1))\leq K$. Then there exist positive constants $C,N$ depending only on  $n,A$ and $K$ such that 
    \begin{equation} \label{1.3}
        \max_{B(0,1/2)}|u|\leq C\sum_{|\alpha|=0}^{N}|D^{\alpha}u(0)| .
    \end{equation}
\end{theorem}

As a corollary, the above theorem holds for the eigenfunctions of the elliptic operator $A$.
\begin{corollary} \label{cor1.4}
    Let $u$ be an eigenfunction of $A$, i.e., $u$ satisifes $\mathrm{div}(A\nabla u)+\mu u=0$ in $B(0,1)\subset \mathbb{R}^{n}$, where $A\in \mathcal{A}_{1}$. Then there exist constants $C,N$ depending only on  $n,A,\mu$ such that 
    \begin{equation} \label{1.4}
        \max_{B(0,1/2)}|u|\leq C\sum_{|\alpha|=0}^{N}|D^{\alpha}u(0)|. 
    \end{equation}
\end{corollary}

Furthermore, using the classical tansformation introduced in Section 3.2 in \cite{Hanlinbook}, we can obtain the general version of the main theorem. In detail, we consider
\begin{equation}\label{general smooth equation}
   \mathcal{L}u\equiv a^{ij}(x)\partial_{ij}u+b^{i}(x)\partial_{i}u+c(x)u=0
\end{equation}
with $a^{ij},b^{i},c\in C^{\infty}(B(0,1))$. And we have the following result.
\begin{theorem}\label{general smooth theorem}
    Let $u$ be a solution to \eqref{general smooth equation} in $B(0,1)\subset \mathbb{R}^{n}$ with \eqref{uniform elliptic}. Assume that $\mathcal{H}^{n-1}(Z_{u}\cap B(0,1))\leq K$, then there exist positive constants $C,N$ depending only on $n,a^{ij},b^{i},c$ and $K$ such that 
    \begin{equation}\label{main estimate}
        \max_{B(0,1/4)}|u|\leq C\sum_{|\alpha|=0}^{N}|D^{\alpha}u(0)| .
    \end{equation}
\end{theorem}
Since the strong unique continuation tells us that the infiniteness of vanishing order implies that $u$ is identically zero. Then returning to our main theorem, a natural question is: Can we replace the right-hand side of \eqref{main estimate} by the term corresponding to vanishing order? However, the answer is negative. We have the following  counterexample.
\begin{example}
    For any $m\geq 1$, consider $u_{m}(x',x_{n})=2m\cdot x_{n}-1$, then it is harmonic in $\mathbb{R}^{n}$ with $|u_{m}(0)|=1$. Furthermore, its zero set in $B(0,1)$ has uniform boound
    \begin{equation}
        \mathcal{H}^{n-1}(Z(u_{m})\cap B(0,1))\leq \omega(n-1),
    \end{equation}
where $\omega(k)$ means the volume of the unit ball in $\mathbb{R}^{k}$.
However, noting that 
\begin{equation}
    \max\limits _{B(0,1/2)}|u|=m-1=(m-1)\cdot|u_{m}(0)|.
\end{equation}
For $m$ can be chosen as any large integer, this gives a negative answer to the above question.
\end{example}
Besides, it is also natural to ask whether this conjecture holds for elliptic PDEs with only $C^{1}$ or Lipschitz coefficients. But for these cases, the higher-order derivatives are not well-defined. So we need different terms to replace formula on the right. In fact, we can try to extend this theorem to the following.

\begin{theorem} \label{thm1.5}
    Let $u$ be a solution to $\mathrm{div}(A\nabla u)=0$ in $B(0,M)\subset \mathbb{R}^{n}$, where $A\in \mathcal{A}_{2}$. Let $M$ be a real number large enough. Assume that there exists a constant $M>0$ such that $\beta(0,r)\leq N$ for all $0<r<M$. Then there exist two constants $T$ and $C$ only depending on $\lambda,n,N$ and $M$ only such that
    \begin{equation}
        \int_{B(0,1/4)}|u|^{2}\leq C\sum_{m=0}^{T}\left<u,\Phi_{m}\right>_{L^{2}(\partial B(0,1))}^{2}.
    \end{equation}
\end{theorem}
Here $\beta(x,r)$ is the frequency function introduced in the next section. The functions $\{\Phi_{m}\}$ are projections to space $H_{m}$ which we will introduce in Section 4. In Section 4, we will also explain why we choose this form.
\begin{remark}
    In Section 4 we prove that Theorem \ref{thm1.5} implies Conjecture 1.3.
\end{remark}
\begin{remark}
    Indeed, we can replace $\beta(0,r)\leq N$ with $\mathcal{H}^{n-1}(Z(u)\cap B(0,M))\leq N$ as in \cite{MR3739232}, Logunov claimed that Theorem \ref{Thm2.7} in our Section 2 is ture for $A\in \mathcal{A}_{2}$.
\end{remark}

Another natural question is how to consider the opposite of this conjecture. In fact, we have the following theorem.
\begin{theorem} \label{inverse theorem}
    Let $u$ be a solution to $\mathrm{div}(A\nabla u)=0$ in $B(0,2)$ with $A\in \mathcal{A}_{1}$. Assume that there exist  constants $N$ and $C$ such that
    \begin{equation} \label{1.5}
        \max_{B(0,2)}|u|\leq C\sum_{|\alpha|=0}^{N}|D^{\alpha}u(0)| 
    \end{equation}
    Then there exist positive constants $r_{0}$ and $K$ such that for any $s\leq r_{0}$,
    \begin{equation}
        \mathcal{H}^{n-1}(Z(u)\cap B(0,s))\leq Ks^{n-1},
    \end{equation}
where $r_{0}$ depends on $n,A$ and $K$ depends on $n,C,N$ and $A$.
\end{theorem}

As we mention before, it is also important to study the interaction between nodal sets and the function itself. Based on this conjecture, we can prove a type of Liouville theorem for elliptic PDEs in divergence form defined on the whole space with constant coefficients. It states that if the volume of nodal sets does not grow too fast, the solution is actually a polynomial. We will see why we cannot deal with general elliptic PDEs in our proof.
\begin{theorem} \label{thm1.7} 
    Let $u$ be a solution to $\mathrm{div}(A\nabla u) =0$ in $\mathbb{R}^{n}$ where $A$ be a constant positive definite matrix. Then there exist constants $K$ and $r_0$ such that
    \begin{equation}
        \mathcal{H}^{n-1}(Z(u)\cap B(0,r))\leq Kr^{n-1}\quad r>r_0
    \end{equation}
    if and only if $u$ is a polynomial.
\end{theorem}
Motivated by the work of F. Lin and Z. Lin, we prove a similar theorem as Theorem 1.3. in \cite{MR4514975}. By  Theorem \ref{thm1.3}, we prove the following theorem. We call it nodal sets comparsion theorem:
\begin{corollary} \label{cor1.8}
    Assume $\mathrm{div}(A\nabla u)=\mathrm{div}(\bar{A}\nabla v)=0$ in $\mathbb{R}^{n}$ with $A,\bar{A}$ are two constant positive definite matrices and $Z(v)\subset Z(u)$. If $u$ is a polynomial, then $v$ is also a polynomial.
\end{corollary}
This corollary has a different form, which is Theorem 1.3. in \cite{MR4514975}. It can prove $u=cv$ for some $c\in\ \mathbb{R}\setminus \{0\}$. But in our version, we only need the condition $Z(v)\subset Z(u)$. Also, we do not need to consider the ratio $u/v$ and the boundary Harnack principle. It is a pity that we cannot deal with more general elliptic PDEs for similar reasons.

Motivated by the above work, we can also give a new proof of Theorem 1.2. in \cite{MR4514975}. Our proof does not rely on the boundary Harnack inequality, but we can only deal with smooth coefficients.
\begin{theorem} \label{thm1.9}
    Suppose that $u$ and $v$ are solutions to $\mathrm{div}(A\nabla u)=0$ and $\mathrm{div}(A'\nabla v)=0$ in $B(0,10)\subset \mathbb{R}^{n}$ with $A,A'\in \mathcal{A}_{1}$. Also, assume that $Z(v)\subset Z(u)$ and $\beta_{u}(0,10)\leq N_{0}<\infty$. Then there exists a constant $D_{0}$ depending only on $n,A,A'$ and $N_{0}$ only such that $\beta_{v}(0,1)\leq D_{0}.$
\end{theorem}
This theorem can be viewed as a local compactness property for a large class of solutions to such elliptic equations (see \cite{MR833393}).\newline

\textbf{Notations.} We use the following notations in this article.
\begin{itemize}
    \item Throughout this article, we denote positive constants by $C$. We may write $C(a_{1},a_{2},\cdots)$ to highlight its dependence on parameters $a_{1},a_{2},\cdots$. The value of $C$ may vary from line to line.
    \item In this article, we denote the ball centered at $x$ with radius $r$ by $B(x,r)=\{y\in \mathbb{R}^{n}:|x-y|\leq r\}$. We will write $B^{n}(x,r)$ if we want to emphasize that the ball belongs to $n$ dimensional space $\mathbb{R}^{n}$.
    \item We denote the $n$-dimensional Hausdorff measure by $\mathcal{H}^{n}$.
    \item For a given function $u$, we denote its nodal set by $Z(u)=\{x:u(x)=0\}$ and denote its singular set by 
$S(u)=\{x:|u(x)|=|\nabla u(x)|=0\}$.
\end{itemize}
\indent \textbf{Acknowledgements.} Jiahuan Li was supported by National Natural Science Foundation of China [grant number 12141105]. Junyuan Wang was partly supported by National Key Research and Development Program of China (No. 2022YFA1005501).

\section{Preliminary}

\subsection{Doubling Index and Frequency Function}

Doubling index and frequency function are two important tools in the study of nodal sets. In this subsection, we will give some definitions and several important theorems. Most of them can be found in the lecture notes by Logunov and Malinnikova\cite{MR4249624}. In this section, we always assume $A\in \mathcal{A}_{1}$.

\begin{definition}
    Let $u$ be a solution to $\mathrm{div}(A\nabla u)=0$ in $B(0,1)$, we can define the doubling index of $u$ as 
    \begin{equation}
        N(x,r):=\log_{2}\frac{\sup_{B(x,2r)}|u|}{\sup_{B(x,r)}|u|}
    \end{equation}
    with $B(x,2r)\subset B(0,1)$. Sometimes, we will write $N(B)=N(x,r)$ if $B=B(x,r)$. Also we will simply write $N(r)$ if $x$ is fixed. If we want to specify the choice of function, we will write $N_{u}(x,r)$.
\end{definition}

The definition is well-defined by the unique continuation property proved by Garofalo and Lin in \cite{MR833393}. There are many important properties and theorems of doubling index. The first one is the almost monotonicity of doubling index.
\begin{lemma}[\cite{MR4249624}]\label{almost monotone of doubling index}
    Let $u$ be a solution to $\mathrm{div}(A\nabla u)=0$ in $B(0,1)$ and let $N(r)=N(0,r)$. Then there exists a constant $C$ depending on $n$ and $A$ such that 
    \begin{equation}
        N(r)\leq CN(R)+C
    \end{equation}
    for all $4r<R<1$.
\end{lemma}

Besides, we can compare doubling index at two different point. The proof is based on the definition and covering argument. As a result, this theorem actually works for all function with well-defined doubling index:
\begin{lemma}\label{changed center estimate}
    Let $u$ be a solution to $\mathrm{div}(A\nabla u)=0$ in $B(0,1)$. Then the following inequality holds
    \begin{equation} \label{2.2}
        N(0,1/16)\leq N(x,1/8)+N(x,1/16)+N(x,1/32) 
    \end{equation}
    for all $x\in B(0,1/32)$.
\end{lemma}

\begin{proof}
    Be the definition, we have
    \begin{equation}
        N(0,1/16)=\log_{2}\frac{\sup_{B(0,1/8)}|u|}{\sup_{B(0,1/16)}|u|}
    \end{equation}
    By the choice of $x$, we have $B(0,1/8)\subset B(x,1/4),B(x,1/32)\subset B(0,1/16)$. As a result
    \begin{align}
 \begin{split}       N(0,1/16)&=\log_{2}\frac{\sup_{B(0,1/8)}|u|}{\sup_{B(0,1/16)}|u|}\\
        &\leq \log_{2}\frac{\sup_{B(x,1/4)}|u|}{\sup_{B(x,1/32)}|u|}\\
        &=N(x,1/8)+N(x,1/16)+N(x,1/32).
        \end{split}
    \end{align}
\end{proof}

We can actually prove $N(x,r)\leq CN(y,R)+C$ if $B(x,2r)\subset B(x,R)$ for some constant $C$ depending on $A,n$ only. But the above version is sufficient to prove our results. Now we will introduce the frequency function. Frquency function was first introduced by Almgren in \cite{almgren1979dirichlet} to study the minimal surface. It was then used to study the unique continuation property and nodal sets of elliptic PDEs. See \cite{MR833393} and \cite{MR1090434} for some applications.

\begin{definition}\label{frequency}
    Let $u$ be a solution to $\mathrm{div}(A\nabla u)=0$ in $B(0,1)$. After chosing a new coordinate, we can assume $A(0)=I$. Then let
    \begin{equation}
        \mu(x)=\frac{(A(x)x,x)}{|x|^{2}}
    \end{equation}
    We can define height function and Dirichlet energy as the following
    \begin{equation}
        h(x,r):=r^{1-n}\int_{\partial B(x,r)}\mu |u|^{2}\mathrm{dS}\quad D(x,r):=r^{1-n}\int_{ B(x,r)}(A\nabla u,\nabla u)\mathrm{dy}
    \end{equation}
    whenever $B(x,r)\subset B(0,1)$.
    The \textit{frequency function} is defined by
    \begin{equation}
        \beta(x,r):=\frac{rD(x,r)}{h(x,r)}
    \end{equation}
    Sometimes, we will write $\beta(r)$ if $x$ is fixed. If we want to specify the choice of function, we will write $\beta_{u}(x,r)$.
\end{definition}
The most improtant property of frequency function is the almost monotonicity property. The proof can be found in the article of Garofalo and Lin :
\begin{theorem}[\cite{MR833393} Theorem 2.1]
    For $u$ as above, there exists a constant $C$ only depening on $n$ and $A$ such that the function $\bar\beta(r)=e^{Cr}\beta(r)$ is non-decreasing.
\end{theorem}
The next question is whether doubling index and frequency function are comparable in some sense. The answer is yes since frequency function is some $L^{2}$ norm and doubling index is $L^{\infty}$ norm. We will state the following lemma:
\begin{lemma}[\cite{MR4249624}]\label{frquency&doubling index}
    There exists a constant $C\geq 1$ depending on $n$ and $A$ such that 
    \begin{equation} \label{2.3}
        C_{1}^{-1}\beta(x,r)-C_{2}\leq N(x,r)\leq C_{1}\beta(x,4r)+C_{2} 
    \end{equation}
    uniformly for all $0<r\leq 1/4$ and all $u$ be a solution to $\mathrm{div}(A\nabla u)=0$ in $B(x,2)$.
\end{lemma}

\subsection{Nodal Sets and Doubling Index} 

In this subsection, we recall a theorem proved by Logunov in \cite{MR3739231}. A direct computation shows that the bounded nodal volume implies bounded doubling index.

\begin{theorem}[\cite{MR3739232} Remark 7.2]\label{Thm2.7}
    Let $u$ be a solution to $\mathrm{div}(A\nabla u)=0$ in $B(0,M)$ with $M$ is a constant large enough depending on $n$ and $A$. We also assume that $A \in \mathcal{A}_1$ and $u(0)=0$. Then we have the following estimate of the volume of the nodal set:
    \begin{equation} \label{2.4}
        \mathcal{H}^{n-1}(\{u=0\}\cap B(0,1))\geq 2^{c\log \beta(0,1/2)/\log\log \beta(0,1/2)} 
    \end{equation}
    for $\beta(0,1/2)\geq \beta_{0}$ and $c$ only depending on $n$ and $A$.
\end{theorem}
We need to mention here that the right-hand side is increasing if we view it as a function of $\beta(r)$ and assume $\beta(r)$ is big enough. So if we assume the nodal volume is bounded above, we can prove that the doubling index and frequency function should also be bounded. But the theorem only works for solutions that vanish at 0. In the next theorem, we will prove that the doubling index is also bounded even if $u(0)\neq 0$.

\begin{theorem}\label{doubling bound}
    Let $u$ be a solution to $\mathrm{div}(A\nabla u)=0$ in $B(0,1)$ and $\mathcal{H}^{n-1}(Z(u))\leq K$. Also, we assume that $Z(u)\cap B(0,1/64)\neq \emptyset$. Then there exists a constant $C$ only depending on $n,A$ and $K$ such that
    \begin{equation}
        N(0,1/32)\leq C.
    \end{equation}
\end{theorem}

\begin{remark}
    Indeed, the doubling index is also bound when $Z(u)\cap B(0,1/64)= \emptyset$. It is a direct result of the Harnack inequality. 
\end{remark}

\begin{proof}

    Case 1: If $u(0)=0$, then by \eqref{2.3}, \eqref{2.4} and the almost monotonicity of frequency function, we have
    \begin{equation}
        N(0,1/32)\leq C\beta(0,1/8)+C\leq C\beta(0,1/2)+C\leq C
    \end{equation}

    Case 2: If $u(0)\neq 0$, then by our assumption, there exists a zero $x\in B(0,1/64)$, then applying Theorem \ref{Thm2.7} to $u$ we have 
    \begin{equation}
        \beta(x,(1-|x|)/2)\leq C(n,A,K).
    \end{equation}
    Combining \eqref{2.2} and \eqref{2.3}, we have the following rescaling version of Lemma 2.3:
    \begin{equation}
        N(0,r)\leq C\beta(x,8r)+C
    \end{equation}
    for all $x\in B(0,r/2)$ and $r\leq 1/16$. Then we let $r=1/32$, we have
    \begin{equation}
        N(0,1/32)\leq C\beta(x,1/4)+C\leq C\beta(x,(1-|x|)/2)+C\leq C.
    \end{equation}
\end{proof}


\section{Nadirashvilli's Conjecture with Smooth Coefficient}

\subsection{Bound in a Smaller Ball}

\indent First, we will prove our main result in a much smaller ball. In this section, we always assume $A\in \mathcal{A}_{1}$, that is, the smooth, uniformly elliptic operator. By Section 3.2, without loss of generality, we may assume $M=1$ in Theorem \ref{Thm2.7}.

\begin{theorem} \label{them 2.8}
    Let $u$ be a solution to $\mathrm{div}(A\nabla u)=0$ in $B(0,1)$ and $\mathcal{H}^{n-1}(Z(u))\leq K$. Then there exist constants $C$ and $N$ only depending on $n,A$ and $K$ such that
    \begin{equation}
        \sup_{B(0,128)}|u|\leq C\sum_{|\alpha|=0}^{N}|D^{\alpha}u(0)|.
    \end{equation}
\end{theorem}

\begin{proof}
    Case 1: If $Z(u)\cap B(0,1/64)= \emptyset$, then by Harnack inequality
    \begin{equation}
        \sup_{B(0,1/128)}|u|\leq C(n,A)|u(0)|.
    \end{equation}
    Case 2: If $Z(u)\cap B(0,1/64)\neq \emptyset$, then by the assimption 
   \begin{equation}
       \mathcal{H}^{n-1}(Z(u)\cap B(0,1))\leq K.
   \end{equation} 
   Therefore we apply Theorem \ref{doubling bound} to obtain that
 \begin{equation}
       N(0,1/32)\leq C(n,A,K).
   \end{equation}
  Without loss of generality, by rescaling, we may assume that $N(0,1/4)\leq C(n,A,K)$ and $\sup\limits_{B(0,1/2)}|u|=1$. Now it suffices for us to prove that there exist constants $C_{0}(n,A,K)$ and $N(n,A,K)$ such that
    \begin{equation}
\sum_{|\alpha|=0}^{N}|D^{\alpha}u(0)|\geq C_{0}.
    \end{equation}
    By $N(0,1/4)\leq C(n,A,K)$ and the almost monotonicity of doubling index, i.e., Lemma \ref{almost monotone of doubling index}, we have 
    \begin{equation}
        N(0,2^{-j})\leq C(n,A)\cdot N(0,1/4)+C(n,A)\leq C_{1}(n,A,K)
    \end{equation} for all $j\geq 2$. Then by definition of doubling index, 
    \begin{equation}
        \frac{\sup\limits_{B(0,2^{1-j})}|u|}{\sup\limits_{B(0,2^{-j})}|u|}\leq 2^{C_{1}}.
    \end{equation}
    Since $\sup\limits_{B(0,1/2)}|u|=1$, we deduce that
    \begin{equation}
        \sup_{B(0,2^{-l})}|u|\geq 2^{-l C_{1}},
    \end{equation}
    where $l\geq 2$ is a integer to be chosen later. Since $A$ is smooth, the standard regularity theory of elliptic equations implies that $u$ is also smooth. Applying Taylor expansion at $0$, we obtain
    \begin{equation}
        u(x)=\sum_{|\alpha|=0}^{N}\frac{1}{\alpha !}|D^{\alpha}u(0)|x^{\alpha}+\sum_{|\beta|=N+1}\frac{1}{\beta !}|D^{\beta}u(\xi_{\beta})|x^{\beta},
    \end{equation}
    where the point $\xi_{\beta}$ lies in the segment connecting $0$ and $x$. Now we take $|x|\leq r=2^{-l-1}$ and get
    \begin{equation}
        2^{-lC_{1}}\leq \sup_{B(0,2^{-l-1})}|u|\leq \sum_{|\alpha|=0}^{N}|D^{\alpha}u(0)|+\sup_{y\in B(0,r),|\beta|=N+1}|D^{\beta}u(y)|2^{-(l+1)(N+1)}\sum_{|\beta|=N+1}\frac{1}{\beta !}.
    \end{equation}
    Then we use the fact that
    \begin{equation}
    (x_{1}+\cdots+x_{n})^{N+1}=\sum_{|\beta|=N+1}\frac{(N+1)!}{\beta!}x^{\beta},
    \end{equation}
    and take $x_{1}=\cdots=x_{n}=1$, we have
    \begin{equation}
        2^{-lC_{1}}\leq \sum_{|\alpha|=0}^{N}|D^{\alpha}u(0)|+\sup_{y\in B(0,1/4),|\beta|=N+1}|D^{\beta}u(y)|2^{-(l+1)(N+1)}\frac{n^{N+1}}{(N+1)!}.
    \end{equation}
    By elliptic interior estimate, we have
    \begin{equation}
        \sup_{y\in B(0,1/4),|\beta|=N+1}|D^{\beta}u(y)|\leq C_{2}(n,A,N) \sup_{B(0,1/2)}|u|=C_{2},
    \end{equation}
    where constant $C_{2}$ depends on $n,\lambda,N$ and the $C^{N}$-norm of $a^{ij}$.\\
    Therefore, combining the above two inequalities, the following estimate holds
    \begin{equation}
        \sum_{|\alpha|=0}^{N}|D^{\alpha}u(0)|\geq 2^{-lC_{1}}-C_{3}(n,A,N)\cdot 2^{-l(N+1)}=2^{-2lC_{1}}(2^{lC_{1}}-C_{3}(n,A,N)\cdot 2^{2lC_{1}-l(N+1)})
    \end{equation}
    where $C_{3}= \frac{n^{N+1}C_{2}}{(N+1)!}$. Now we choose $N+1=2C_{1}(n,A,K)$, then $C_{3}$ will only depend on $n,A$ and $K$. Furthermore we can choose constant $l=l(n,K,A)$ such that 
    \begin{equation}
        2^{l\cdot C_{1}(n,A,K)}\geq 2C_{3}(n,A,K).
    \end{equation}
Then we can conclude that
\begin{equation}
\sum_{|\alpha|=0}^{N}|D^{\alpha}u(0)|\geq 2^{-2lC_{1}}C_{3}:=C_{0}(n,A,K).
    \end{equation}
\end{proof}

\subsection{Propagation of the Smallness Results}

In this subsection, we will finish our proof with the help of propagation of the smallness results proved by Logunov and Malinnikova in \cite{MR4249624}. 
\begin{lemma}\label{propagation of smallness}
    Let $u$ be a solution to $\mathrm{div}(A\nabla u)=0$ in $B(x,2r)$. For any subset $E\subset B(x,r)$ with positive Lebesgue measure, we have 
    \begin{equation} \label{4.1}
        \sup_{B(x,r)}|u|\leq C_{1}\sup_{E}|u|(C_{1}\frac{|B(x,r)|}{|E|})^{C_{1}N(B(x,r))} 
    \end{equation}
    for some $C_{1}$ only depending on $n$ and $A$.
\end{lemma}

This theorem states that we can propagate the bound to a larger domain once we know the doubling index is bounded. We need to mention here that the doubling index bound in Theorem \ref{them 2.8} also holds for any point.

Now based on Theorem \ref{them 2.8} and the above theorem, we can finish the proof of Theorem \ref{thm1.3}.
\begin{proof}
    Define the set $P$ as the following:
    \begin{equation}
        P:=\{r\in (0,\frac{3}{4}):\max_{B(0,r)}|u|\leq C\sum_{|\alpha|=0}^{N}|D^{\alpha}u(0)| \text{ holds for some }C,N \text{ depending on } n,\lambda,  K\}
    \end{equation}
    It is clear that $P\neq \emptyset$ and $1/128\leq \sup\limits_{r\in P}r<3/4$. If $r_{0}\in P$, then for all $r<r_{0}$, we have $r\in P$. We will prove $1/2\in P$. Now assuming that $r_{0}\in P$ and fixing a constant $r$ small enough with $r_{0}+36r<1$, we claim that $r_{0}+r\in P$. If the claim holds, after a finite-steps process, we can easily conclude that $1/2\in P$.\\
    
    Now we take any point $x\in \partial B(0,r_{0})$ and let $E=B(x,r)\cap B(0,r_{0})$. By Lemma \ref{propagation of smallness}, we have
    \begin{equation}
        \sup_{B(x,r)}|u|\leq C_{1}\sup_{E}|u|(C_{1}\frac{|B(x,r)|}{|E|})^{C_{1}N(x,r)}.
    \end{equation}
Next we prove  
\begin{equation}
        \sup_{B(x,r)}|u|\leq C\sup_{E}|u|\leq C\sum_{|\alpha|=0}^{N}|D^{\alpha}u(0)|
    \end{equation}
Therefore, by a compactness argument, we obtain $r+r_{0}\in P$. We will consider the following two cases.
   Case 1: If $B(x,2r)\cap Z(u)=\emptyset$, by Harnack inequality, we have
    \begin{equation}
        \sup_{B(x,r)}|u|\leq C |u(x)|\leq C\sup_{E}|u|\leq C\sum_{|\alpha|=0}^{N}|D^{\alpha}u(0)|
    \end{equation}
    Case 2: If $B(x,2r)\cap Z(u)\neq\emptyset$, we may assume that $y$ is in this set. Then $B(y,1-2r-r_{0})\subset B(0,1)$. Therefore, by $\mathcal{H}^{n-1}(Z(u)\cap B(0,1))\leq K$ and Theorem 1.2 in \cite{MR4814921}, i.e., the almost sharp lower bound estimate, we have
    \begin{equation}
      N(y,\frac{1-2r-r_{0}}{8})\leq C(n,A,K).
    \end{equation}
    By definition, we have the following inequalities
    \begin{align}
    \begin{split}       N(x,3r)&=\log_{2}\frac{\sup_{B(x,6r)}|u|}{\sup_{B(x,3r)|u|}}\\
        &\leq \log_{2}\frac{\sup_{B(y,8r)}|u|}{\sup_{B(y,r)}|u|}\\
        &=N(y,r)+N(y,2r)+N(y,4r)\leq C(n,A,K).\\
        \end{split}
    \end{align}
    where in the last inequality, we use the assumption $r_{0}+36r< 1$ and the Lemma \ref{almost monotone of doubling index}. Therefore we have $N(x,r)\leq C(n,A,K)$.
    Then we use propagation of the smallness results to obtain the following estimate
    \begin{equation}
        \sup_{B(x,r)}|u|\leq C_{1}\sup_{E}|u|(C_{1}\frac{|B(x,r)|}{|E|})^{C_{1}N(x,r)}\leq C \sum_{|\alpha|=0}^{N}|D^{\alpha}u(0)|.
    \end{equation}
    As a result, we have $r+r_{0}\in P$ and furthermore we get $\frac{1}{2}\in P$.
\end{proof}
\subsection{General Cases and The Inverse Problem}

Next, we will prove the Corollary \ref{cor1.4}. It is trivial with a standard trick and an already-known result.

\begin{proof}[Proof of Corollary \ref{cor1.4}]
    Let $h(x,t)=u(x)e^{\sqrt{\mu}t}$ with $t\in \mathbb{R}$ and $A'(x,t)=\mathrm{diag}\{A(x),1\}$. Then we can check that $h$ is the solution to the equation $\mathrm{div}(A'\nabla h)=0$ defined on the domain $B^{n}(0,1)\times \mathbb{R}$. By Logunov's result in \cite{MR3739231}, we know that
    \begin{equation}
        \mathcal{H}^{n-1}(Z(u)\cap B^{n}(0,1))\leq C\mu^{\alpha}
    \end{equation}
    As a result, for a fixed $\mu$, we have
    \begin{equation}
        \mathcal{H}^{n}(Z(h)\cap B^{n+1}(0,1))\leq C\mu^{\alpha}
    \end{equation}
    Recall that $B^{n+1}(0,1)$ is the unit ball in $n+1$ dimensional space $\mathbb{R}^{n+1}$. By Theorem \ref{thm1.3}, we get
    \begin{equation}
        \sup_{B^{n+1}(0,1/2)}|h|\leq C\sum_{|\alpha|=0}^{N}|D^{\alpha}h(0)|
    \end{equation}
    A simple computation shows that
    \begin{equation}
        D^{\alpha}h=\sum_{\beta\leq\alpha}(\sqrt{\mu})^{|\beta|}D^{\alpha-\beta}u\cdot e^{\sqrt{\mu}t}
    \end{equation}
    And also
    \begin{equation}
\sup_{B^{n+1}(0,1/2)}|h|=\sup_{|x|^{2}+t^{2}\leq 1/4}|u(x)e^{\sqrt{\mu}t}|\geq \sup_{|x|^{2}+t^{2}\leq 1/4}|u(x)|=\sup_{B^{n}(0,1/2)}|u(x)|
    \end{equation}
    Combining the above two inequalities, we have
    \begin{equation}
        \max_{B^{n}(0,1/2)}|u|\leq C\sum_{|\alpha|=0}^{N}|D^{\alpha}u(0)|
    \end{equation}
    for some $C,N$ depending on $\mu,A,n$.
    
\end{proof}

Now as a corollary of main theorem, we can also deal with the general elliptic equations with smooth coefficients.
\begin{proof}[Proof of Theorem \ref{general smooth theorem}]
    Here we apply the transformation introduced in \cite{Hanlinbook}, which helps to transform the general elliptic equation into the divergence form. Therefore we can apply the above theorem.\\
\indent In details, for $(x,x_{n+1},x_{n+2})$ we defiene
\begin{equation}
    v(x,x_{n+1},x_{n+2})=(2-x_{n+2})(2-x_{n+1})u(x).
\end{equation}
Then a direct calculation implies that $v$ solves the equation
\begin{equation}
    \bar{a}^{ij}\partial_{ij}v=0\;\text{in}\;B^{n+2}(0,1)\subset \mathbb{R}^{n+2}.
\end{equation}
Where 
\begin{equation}
    [\bar{a}^{ij}]=\left[
    \begin{matrix}
        (a^{ij})_{n\times n} & 0 & -\frac{b_{i}}{2}(2-x_{n+2})\\
        0 &1 & \frac{(2-x_{n+2})(2-x_{n+1})\cdot c}{2}\\
         -\frac{b_{i}}{2}(2-x_{n+2})& \frac{(2-x_{n+2})(2-x_{n+1})\cdot c}{2} &M_{1}
    \end{matrix}
    \right]_{(n+2)\times (n+2)}
\end{equation}
with the positive constant $M_{1}$ that is large enough to ensure the uniformly ellipicity. 
Now we denote $\bar{x}=(x,x_{n+1},x_{n+2})$, for $(\bar{x},x_{n+3})\in \mathbb{R}^{n+3}$, we define the function
\begin{equation}
    w(\bar{x},x_{n+3})=v(\bar{x})
\end{equation} 
Then $w$ satisfies the equation
\begin{equation}
    \partial_{i}(\hat{a}^{ij}\partial_{j}w)=0 \;\text{in}\;B^{n+3}(0,1)\subset \mathbb{R}^{n+3}.
\end{equation}
Where \begin{equation}
    [\hat{a}^{ij}]=\left[
    \begin{matrix}
        (\bar{a}^{ij})_{(n+2)\times (n+2)} & x_{n+3}\bar{b}\\
        x_{n+3}\bar{b}^{T} & M_{2}\\
        
    \end{matrix}
    \right]_{(n+3)\times (n+3)}
\end{equation}
with $\bar{b}_{i}=-\partial_{i}\bar{a}^{ij}$ and the constant $M_{2}>0$ is also large enough to ensure the uniformly ellipticity of $\hat{a}^{ij}$.\\
\indent Noting that $w(x,x_{n+1},x_{n+2},x_{n+3})=0$ if and only if $u(x)=0$ for $x\in B(0,1)\subset \mathbb{R}^{n}$. Therefore, a simple inclusionship implies that
\begin{equation}
    \mathcal{H}^{n+2}(Z(w)\cap B^{n+3}(0,1))\leq 100\mathcal{H}^{n-1}(Z(u)\cap B^{n}(0,1))\leq 100K.
\end{equation}
The last inequality uses our assumption. Since $\hat{a}^{ij}$ are all smooth, then applying Theorem \ref{thm1.3} to $w$, there exist two positive constants $C,N$ depending only on  $n,a^{ij},b^{i},c$ and $K$ such that 
    \begin{equation} 
        \max_{B^{n+3}(0,1/2)}|w|\leq C\sum_{|\alpha|=0}^{N}|D^{\alpha}w(0)| .
    \end{equation}
Recall that $w(x,x_{n+1},x_{n+2},x_{n+3})=(2-x_{n+2})(2-x_{n+1})u(x)$ and the fact that
\begin{equation}
   \{x\in B^{n}(0,1/4),|x_{n+1}|,|x_{n+2}|,|x_{n+3}|\leq 1/4\} \subset B^{n+3}(0,1/2).
\end{equation}
Then we can conclude that
\begin{equation} 
        \max_{B^{n}(0,1/4)}|u|\leq C\sum_{|\alpha|=0}^{N}|D^{\alpha}u(0)| .
    \end{equation}
\end{proof}

\begin{remark}
    Above two proofs illustrate two different ways to transform general elliptic PDEs into divergence form. That is why we prove these two theorems independently. Also, the first proof shows that the constant $N$ can de independent of $\mu$.
\end{remark}

Now we come to the proof of Theorem \ref{inverse theorem}.
\begin{proof}[Proof of  Theorem \ref{inverse theorem}]
    By our assumption and elliptic interior estimate, we have \begin{equation}
        \max\limits_{B(0,2)}|u|\leq K(n,A,C,N)\max\limits_{B(0,1)}|u|,
    \end{equation}
    which by definition deduces the inequality $N(0,1)\leq K(n,A,C,N)$.\\
\indent By Lemma \ref{frquency&doubling index} , Lemma \ref{almost monotone of doubling index} and the almost monotonicity of doubling index, we have  
\begin{equation}
   \beta:=\beta(0,s)\leq C+C\cdot N(0,s)\leq C+C\cdot N(0,1)\leq K(n,A,C,N).
\end{equation}
Therefore by Theorem 1.4 in \cite{MR3688031}, the size of nodal set has an exponential upper bound with respect to frequency function, then for any $s\leq r_{0}$
\begin{equation}
    \mathcal{H}^{n-1}(Z(u)\cap B(0,s))\leq (C(n,A)\beta)^{\beta}s^{n-1}\leq K(n,A,C,N)s^{n-1}.
\end{equation}

\end{proof}
\section{A Version of Nadirashvili's Conjecture with $C^{1}$ Coefficients}

In this section, we will prove a version of Nadirashvili's conjecture with only $C^{1}$ coefficients. In this case, the derivatives may not be well-defined. Motivated by unique continuation property, we will use the $L^{2}$ norm.

\begin{definition}
$S=\partial B_1$, $H_m\left( S \right) :=\left\{ f|_S, f \text{ is homogeneous harmonic polynomial of degree }m \right\}$ 
\end{definition}
Noting that by the standard theory of spherical harmonics, see Theorem 5.12 in \cite{axler2013harmonic}, we have
\begin{equation}
    L^{2}(S)=\bigoplus_{m=0}^{\infty}H_{m}(S).
\end{equation}
\begin{proof}[Proof of Theorem \ref{thm1.5}]
    For convenience, we will let $S:=\partial B(0,1),L^{2}(S)=\bigoplus_{m=0}^{\infty}H_{m}(S)$. Assume that $u|_{S}=\sum_{i=0}^{\infty}\Phi_{m}$ in the $L^{2}$ sense and $\Phi_{m}\in H_{m}(S)$. We will consider the following Dirichlet problem:
    \begin{align*}
        \Delta v=0\quad &\text{in } B(0,1)\\
        v=u \quad &\text{on }\partial B(0,1)
    \end{align*}
    Let $g(t)=\int_{\partial B(0,1)}|\nabla(v+t(v-u))|^{2}$ with $t\in[0,1]$. Then, we take derivative on the both size to get
    \begin{align}
    \begin{split}
        g'(t)&=\int_{B(0,1)}2\nabla v\cdot\nabla(v-u)+2t\int_{B(0,1)}|\nabla(v-u)|^{2}\\
        &\geq \int_{B(0,1)}2\nabla v\cdot\nabla(v-u)\\
        &=\int_{B(0,1)}2\Delta v\cdot(v-u)+\int_{\partial B(0,1)}2\frac{\partial v}{\partial n}(v-u)\\
        &=0
       \end{split} 
    \end{align}
    Therefore, we know that the function $g(t)$ is increasing. So we get
    \begin{equation}
        \int_{B(0,1)}|\nabla u|^{2}\geq \int_{B(0,1)}|\nabla v|^{2}
    \end{equation}
    Now we let $a_{d}:=\left<u,\Phi_{m}\right>_{L^{2}(S)}$ and $f_{k}:=\sum_{m=0}^{k}\Phi_{m}$. By definition, we get
    \begin{equation} \label{5.1}
        \int_{\partial B(0,1)}|v|^{2}=\sum_{d=0}^{\infty}a_{d}^{2} 
    \end{equation}
    and 
    \begin{equation}
        \lim_{k\rightarrow \infty}||v-f_{k}||_{L^{2}(S)}=0
    \end{equation}
    Since $v$ and $f_{k}$ are all harmonic functions, by standard elliptic estimate, we get
    \begin{equation}
        \sup_{B(0,1-\varepsilon)}|v-f_{k}|\leq C||v-f_{k}||_{L^{2}(S)}
    \end{equation}
    for a fixed $\varepsilon\in (0,1)$. So we have
    \begin{equation*}
        \lim_{k\rightarrow \infty}\sup_{B(0,1-\varepsilon)}|v-f_{k}|=0
    \end{equation*}
    By harmonicity, we have
    \begin{equation}
        \int_{B(0,1)}|\nabla v|^{2}\geq  \int_{B(0,1-\varepsilon)}|\nabla v|^{2}=\lim_{k\rightarrow \infty}\int_{B(0,1-\varepsilon)}|\nabla f_{k}|^{2}=\sum_{d=0}^{\infty} da_{d}^{2}(1-\varepsilon)^{2d+n-2}
    \end{equation}
    Since $\varepsilon$ is arbitrary, we now let $\varepsilon\rightarrow 0$ to get
    \begin{equation} \label{5.2}
        \int_{B(0,1)}|\nabla u|^{2}\geq \int_{B(0,1)}|\nabla v|^{2}\geq \sum_{d=0}^{\infty}da_{d}^{2} 
    \end{equation}
    By the definition of frequency function, we know that
    \begin{equation}
        \int_{B(0,1)}|\nabla u|^{2}\leq N_{0}\int_{\partial B(0,1)}|u|^{2}
    \end{equation}
    for some integer $N_{0}=N_0 (N)$ big enough. Then by \eqref{5.1} and \eqref{5.2}, we have
    \begin{align}
    \begin{split}
        \sum_{d=0}^{\infty}da_{d}^{2}&\leq N_{0}\sum_{d=0}^{\infty}a_{d}^{2}\\
        \sum_{d=N_{0}+1}^{\infty}(d-N_{0})a_{d}^{2}\leq &\sum_{d=0}^{N_{0}-1}(N_{0}-d)a_{d}^{2}\leq N_{0}\sum_{d=0}^{N_{0}}a_{d}^{2}
        \end{split}
    \end{align}
    Then we have
    \begin{equation}
        \int_{B(0,1/4)}|u|^{2}\leq C\int_{\partial B(0,1)}|u|^{2}\leq C\sum_{d=0}^{\infty}a_{d}^{2}\leq C\sum_{d=0}^{N_{0}}a_{d}^{2}
    \end{equation}
    where the constant $C$ in the last formula depending on $N$.
    
\end{proof}
\begin{proof}[Proof of Conjecture 1.3 by Theorem \ref{thm1.5}]Without loss of generality, we assume that
$$\,\, u=\sum_{m=0}^{\infty}{\varPhi _m}\,\,\text{in}\,\,B_1 ,\,\,\varPhi _m \text{ is hhp of degree m}. $$

$$H_m:=\left\{ \text{all hhp of degree m}\right\}.
\left< f,g \right> :=\int_{\partial B_1}{fg\,\,\mathrm{dS}}$$
Here \textit{hhp} means the homogeneous harmonic polynomial. Equipped with the above inner product $\left< \,, \right>$, the linear space $H_m$ becomes a finite dimensional Hilbert space. Then we consider the following standard norm for polynomials.
$$\left\| f \right\| :=\sum_{|\alpha |=m}{|D^{\alpha}f\left( 0 \right) |}$$
Equipped with the above norm $\left\| \cdot \right\| $, the space $H_m$ becomes a finite dimensional Banach space. Since any two norms on finite dimensional Banach space are equivalent, there exists a constant $C$ only depending on $m$ such that
\begin{equation}
C^{-1}\left\| \varPhi \right\| \leqslant \sqrt{\left< \varPhi ,\varPhi \right>}\leqslant C\left\| \varPhi \right\| \,\,, \forall \varPhi \in H_m
    \end{equation}
So by Theorem \ref{thm1.5} we have the following estimate
    \begin{align}
        \begin{split}
        \underset{B(0,1/8)}{\sup}|u|^{2}& \leq C\int_{B(0,1/4)}{u^2}\leq C\sum_{m=0}^T{\left< u,\varPhi _m \right> =}C\sum_{m=0}^T{\left< \varPhi _m,\varPhi _m \right>}\\
        &\leq C\sum_{m=0}^{T}\left( \sum_{|\alpha|=m} |D^{\alpha}\varPhi _m\left( 0 \right) | \right)^{2}=C\sum_{m=0}^{T}\left( \sum_{|\alpha|=m} |D^{\alpha}u\left( 0 \right) | \right)^{2}\\
        &\leq C\sum_{m=0}^{T}\sum_{|\alpha|=m}|D^{\alpha}u(0)|^{2}
        \end{split}
    \end{align}
    Therefore we have proven Conjecture 1.3 by Theorem \ref{thm1.5}.
\end{proof}
\section{Several Applications}

\subsection{Liouville-type Theorem and its Application}

First, we will prove Theorem \ref{thm1.7}. It is based on the standard elliptic estimates.

\begin{proof}[Proof of Theorem \ref{thm1.7}]
    $``\Longrightarrow"$
    
   Noting that $A$ is a constant positive matrix, and we denote $v(x):=u(rx)$. Then $v$ also satisfies the equation $\mathrm{div}(A\nabla v)=0$ and 
    \begin{equation}
        \mathcal{H}^{n-1}(Z(v)\cap B(0,1))\leq K.
    \end{equation}
    Therefore, there exist constants $C,N$ such that
    \begin{equation}
        \sup_{B(0,1/2)}|v|\leq C\sum_{|\alpha|=0}^{N}|D^{\alpha}v(0)|
    \end{equation}
    By our definition of $v$, we get
    \begin{equation}
        \sup_{B(0,r/2)}|u|\leq C\sum_{i=0}^{N}\sum_{|\alpha|=i}|D^{\alpha}u(0)|r^{i}
    \end{equation}
    By gradient estimate, we have
    \begin{equation} \label{inequality}
        \sup_{B(0,r/4)}\sum_{|\alpha|=N+1}|D^{\alpha}u|\leq  \frac{C}{r^{N+1}}\sup_{B(0,r/2)}|u|\leq\frac{C}{r^{N+1}}\sum_{i=0}^{N}\sum_{|\alpha|=i}|D^{\alpha}u(0)|r^{i}.
    \end{equation}
    Therefore, we have the following limit exists
    \begin{equation}
        \small{\underset{r\rightarrow \infty}{\lim}\underset{B\left( 0,r/4 \right)}{\sup}\sum_{|\alpha |=N+1}{|}D^{\alpha}u|=0}.
    \end{equation}
    Hence $u$ is polynomial with degree no more than $N$ since $u$ is smooth.\newline
    $``\Longleftarrow "$
    Now, we assume that $u$ is a harmonic polynomial. Let $m:=\mathrm{deg} \,u$, then we have
    \begin{equation}
        \beta(0,r)\equiv m,\forall r>0.
    \end{equation}
    Then by the classical upper bound of nodal set, with respect to frequency function, we have
    \begin{equation}
        \mathcal{H}^{n-1}(Z(u)\cap B(0,r))\leq Kr^{n-1}.
    \end{equation} 
    where $K=K(n,A,m)$.
\end{proof}

\begin{remark}
    If we consider general PDEs $\mathrm{div}(A(x)\nabla u)=0$, then by our definition, $v$ should statisfies a different equation $\mathrm{div}(A(rx)\nabla v)=0$. But the constants $C$ and $N$ in \eqref{inequality} depend on the components of $A$. Our proof fails for this case as the limit may not exists.
\end{remark}

Next, we will prove Corollary \ref{cor1.8}. It follows immediately from the above theorem.
\begin{proof}[Proof of Corollary \ref{cor1.8}]
    By the fact that $u$ is a polynomial and Theorem \ref{thm1.7}, we know that there exist constants $K$ and $r_0$ such that
    \begin{equation}
        \mathcal{H}^{n-1}(Z(u)\cap B(0,r))\leq Kr^{n-1}\quad r>r_0
    \end{equation} 
    By our assumption $Z(v)\subset Z(u)$, we know that for constants $K$ and $r_0$ above, we have the following estimate
    \begin{equation}
        \mathcal{H}^{n-1}(Z(v)\cap B(0,r))\leq Kr^{n-1}\quad r>r_0
    \end{equation}
    By Theorem \ref{thm1.7}, we know that $v$ is a polynomial.
\end{proof}

\subsection{Comparsion Theorem}

In this subsection, we will prove Theorem \ref{thm1.9}. We emphasize again that $A$ and $A'$ can be two different matrices in $\mathcal{A}_{1}$. The constants depend on the matrices since we need to use gradient estimate.

\begin{proof}[Proof of Theorem \ref{thm1.9}]
    Since $\beta_{u}(0,10)\leq N_{0}<\infty$, we have the following estimate
    \begin{equation}
        \mathcal{H}^{n-1}(Z(u)\cap B(x,r))\leq K
    \end{equation}
    for constant $K$ only depending on $N_{0}$ and $A$. As a result, the above estimate also holds for the nodal set of $v$. By Theorem \ref{thm1.3}, we know that there exist constants $C$ and $N$ such that
    \begin{equation}
        \sup_{B(0,2)}|v|\leq C\sum_{|\alpha|=0}^{N}|D^{\alpha}v(0)|\leq D\sup_{B(0,1)}|v|
    \end{equation}
    where in the last inequallity, we use standard gradient estimate and the fact that $D$ is a constant depending on $K$ and $A'$. By definition of doubling index and Lemma \ref{2.3}, we know that
    \begin{equation}
        \beta_{v}(0,1)\leq D_{0}
    \end{equation}
    for some constant $D_{0}$.
\end{proof}

\section{The Relationship between Two Conjectures}

In the last section, for the completeness of this paper, we prove the fact that the Conjecture 1.3 actually implies Conjecture 1.4.

\begin{proof} 
Now we assume that Conjecture 1.3 holds and fix a constant $K>0$.\\
Case 1: If $\mathcal{H}^{n-1}\left( Z\left( u \right) \cap B\left( 0,1 \right) \right) \geq K$, then Conjecture 1.4 holds.\\
Case 2: If $\mathcal{H}^{n-1}\left( Z\left( u \right) \cap B\left( 0,1 \right) \right) \leq K$, then by Conjecture 1.3, we have the following bound of doubling index:
    \begin{equation}
        N_{u}(0,1/4)\leq C.
    \end{equation}
Hence, by the classical lower bound of nodal set, we obtain
    \begin{equation}
        \mathcal{H}^{n-1}\left( Z\left( u \right) \cap B\left( 0,1 \right) \right) \geq \mathcal{H}^{n-1}\left( Z\left( u \right) \cap B\left( 0,1/2 \right) \right) \geq CN_{u} \left( 0,1/4 \right) ^{1-n}\geq C.
    \end{equation}
\end{proof}

\small
\bibliographystyle{alpha}
\bibliography{reference}
{\small 
\indent (Jiahuan Li) SCHOOL OF MATHEMATICS SCIENCE, UNIVERSITY OF SCIENCE AND TECHNOLOGY OF CHINA, HEFEI, 230022, CHINA.\;
Email.address: jiahuan@mail.ustc.edu.cn\\ \\
(Junyuan Wang) SCHOOL OF MATHEMATICAL SCIENCES, ZHEJIANG UNIVERSITY, HANGZHOU, 310058, CHINA.\;
Email address: wangjunyuan@zju.edu.cn \\ \\
(Zhichen Ying) SCHOOL OF MATHEMATICAL SCIENCES, ZHEJIANG UNIVERSITY, HANGZHOU, 310058, CHINA.\;
Email address: yingzc059@gmail.com
}

\end{document}